\documentclass[amsfonts,12pt]{amsart}

\usepackage{epsfig}

\newtheorem{thm}{Theorem}[section]

\newtheorem{lem}[thm]{Lemma}
\newtheorem{cor}[thm]{Corollary}

\newtheorem{prp}[thm]{Proposition}
\newtheorem{ex}[thm]{Example}

\newcommand{\conv}{{\mathrm{conv}}\,}

\newcommand{\aff}{{\mathrm{aff}}\,}

\newcommand{\ee}{\varepsilon}

\newcommand{\R}{\mathbb{R}}

\newcommand{\N}{\mathbb{N}}

\newcommand{\cK}	{\mathcal K}

\newcommand{\cZ}	{\mathcal Z}
\newcommand{\Sphere}[1][n-1]	{S^{#1}}

\begin{document}
\bigskip

\title{A CHARACTERIZATION OF BLASCHKE ADDITION}
\author[Richard J.~Gardner, Lukas Parapatits, and Franz E.~Schuster]
{Richard J.~Gardner, Lukas Parapatits, and Franz E.~Schuster}
\address{Department of Mathematics, Western Washington University,
Bellingham, WA 98225-9063,USA} \email{Richard.Gardner@wwu.edu}
\address{Institute of Discrete Mathematics and Geometry, Vienna University of Technology, Wiedner Hauptstrasse 8--10/104, 1040 Vienna, Austria}
\email{lukas.parapatits@tuwien.ac.at}
\address{Institute of Discrete Mathematics and Geometry, Vienna University of Technology, Wiedner Hauptstrasse 8--10/104, 1040 Vienna, Austria} \email{franz.schuster@tuwien.ac.at}
\thanks{First author supported in
part by U.S.~National Science Foundation Grant DMS-1103612. Second author supported by Austrian Science Fund (FWF) Grant P23639-N18. Third author supported by FWF Grant P22388-N13. 
Second and third authors also supported by European Research Council (ERC) Grant 306445.}
 \maketitle
\begin{abstract}
A characterization of Blaschke addition as a map between origin-symmetric convex bodies is established.  This results from a new characterization of Minkowski addition as a map between origin-symmetric zonoids, combined with the use of L\'{e}vy-Prokhorov metrics.  A full set of examples is provided that show the results are in a sense the best possible.
\end{abstract}

\section{Introduction}

Like so much else in convex geometry, the operation between convex bodies now called Blaschke addition goes back to Minkowski \cite[p.~117]{Min97}, at least when the bodies are polytopes.  Given convex polytopes $K$ and $L$ in $\R^n$, a new convex polytope $K\,\sharp\,L$, called the Blaschke sum of $K$ and $L$, has a facet with outer unit normal in a given direction if and only if either $K$ or $L$ (or both) do, in which case the area (i.e., $(n-1)$-dimensional volume) of the facet is the sum of the areas of the corresponding facets of $K$ and $L$.  Blaschke \cite[p.~112]{Bla56} found a definition suitable for smooth convex bodies in $\R^3$.  The modern definition, appropriate for any pair of convex bodies, had to wait for the development of surface area measures and is due to Fenchel and Jessen \cite{FenJ38}.  They defined the surface area measure of $K\,\sharp\,L$ to be the sum of the surface area measures of $K$ and $L$, and this determines the Blaschke sum, up to translation. (See (\ref{Bladef}) below.  The existence of $K\,\sharp\,L$ is guaranteed by Minkowski's existence theorem, a classical result that can be found, along with definitions and terminology, in Section~\ref{prelim}.)

After Minkowski addition, with which it coincides, up to translation, in the plane, and perhaps $L_p$ addition (the natural $L_p$ extension of Minkowski addition), Blaschke addition is the most important operation between sets in convex geometry.  Its fundamental nature is evidenced by the fact that every convex polytope is a finite Blaschke sum of simplices and every $o$-symmetric (i.e., symmetric with respect to the origin) convex polytope is a finite Blaschke sum of parallelotopes; see \cite{FirG} or \cite[pp.~334--5]{Gru03}.  In a similar vein, Grinberg and Zhang \cite{GriZ99} showed that each $o$-symmetric convex body is a limit in the Hausdorff metric of finite Blaschke sums of ellipsoids.

Blaschke addition has found many applications in geometry.  It was employed by Petty \cite{Pet67} and Schneider \cite{Sch67} in their independent solutions of Shephard's problem on areas of orthogonal projections (i.e., brightness functions) of centrally symmetric convex bodies.  Indeed, it is a natural tool generally when brightness functions arise, because the brightness function of the Blaschke sum $K\,\sharp\,(-K)$ of a convex body $K$ and its reflection $-K$ in the origin is, up to a constant multiple, equal to that of $K$ itself.  In this context, see, for example, \cite{Gro98}, \cite{How06}, and \cite{Kla10}.  Blaschke sums also appear in the theory of valuations, as in \cite{Hab11}, \cite{Kla00}, \cite{Lud12}, \cite{McM80}, \cite{SchS}, \cite{Schu07b}, and \cite{Schu08}; the theory of random sets and processes, as in \cite{GKW98}, \cite{HugS}, and \cite[p.~200]{Mol05}; Minkowski geometry \cite[Chapter~6]{Tho96}; affine isoperimetric inequalities, in particular, the Kneser-S\"{u}ss inequality \cite[Theorem~7.1.13]{Sch93}; and the study of affine surface area \cite{Lut91}, decomposability \cite{CCG98}, and additive maps of convex bodies \cite{Kid06}, \cite{Schu07}.  This list is by no means comprehensive and further references can be found in \cite[Notes~3.3 and 3.4]{Gar06} and \cite[pp.~395--396]{Sch93}.

The concept of a surface area measure of a convex body is of interest outside the mathematics community.  Computer scientists and electrical engineers, for example, tend to refer to it as the extended Gaussian image.  This is the term used by Zouaki \cite{Zou03a}, \cite{Zou03b} who applies Blaschke addition to object metamorphosis in computer vision (animation) and computer aided design.

Gardner, Hug, and Weil \cite{GHW} initiated a program of research with the goal of characterizing the important operations between sets in geometry.  The aim is to single out such an operation as the only one satisfying a short list of fundamental properties.  Results of this type are not merely a matter of axiomatics, but yield extremely potent geometrical information.  For example, \cite[Corollary~9.11]{GHW} states that an operation $*$ between compact convex sets in $\R^n$, $n\ge 2$, is continuous in the Hausdorff metric, $GL(n)$ covariant, and has the identity property, if and only if it is Minkowski addition.  (Here, $GL(n)$ covariant means that the operation can take place before or after the sets concerned undergo the same transformation in $GL(n)$, with the same effect.   The identity property says that the operation leaves a set unchanged if the other set is the single point at the origin; see Section~\ref{properties} for formal definitions.)  Other results in \cite{GHW} successfully characterize $L_p$ addition.  The methods used in \cite{GHW} require continuity in the Hausdorff metric in the class of (not necessarily full-dimensional) compact convex sets.  Unfortunately, by \cite[Theorem~5.3]{GHW}, when $n\ge 3$, Blaschke addition cannot even be extended to a continuous operation between $o$-symmetric compact convex sets.

It is therefore clear that new techniques must be introduced in order to provide a characterization of Blaschke addition and we do this here. Our main result is as follows.

\smallskip

{\bf Theorem A}. {\em If $n\ge 3$, then an operation between $o$-symmetric convex bodies in $\R^n$ is uniformly continuous in the L\'{e}vy-Prokhorov metric $\delta_{LP}$ and $GL(n)$ covariant if and only if $K*L=aK\,\sharp\,bL$, for some $a,b\ge 0$ and all $o$-symmetric convex bodies $K$ and $L$.}

\smallskip

See Theorem~\ref{thmBlaschke} for a more precise formulation.  In the statement of Theorem~A it should be understood that $K*L=bL$ if $a=0$ and $K*L=aK$ if $b=0$.  A characterization of Blaschke addition (the case $a=b=1$ of Theorem~A) follows quickly when one extra weak property, a suitable modification of the identity property, is added; see Corollary~\ref{characterization Blaschke}. The L\'{e}vy-Prokhorov distance $\delta_{LP}(K,L)$ between convex bodies $K$ and $L$ with centroids at the origin is defined to be the usual L\'{e}vy-Prokhorov distance between their surface area measures.  Formal definitions of L\'{e}vy-Prokhorov metrics and some supplementary material concerning them is provided in Section~\ref{LevProk}.

The underlying idea behind Theorem~A is to employ the well-known connection between Minkowski addition and Blaschke addition via the projection body operator $\Pi$.  If $K$ is a convex body in $\R^n$, its projection body $\Pi K$ has support function equal to the brightness function of $K$.  The connection then takes the form
\begin{equation}\label{pik}
\Pi(K\,\sharp\, L)=\Pi K + \Pi L,
\end{equation}
where $+$ denotes Minkowski addition; see, for example, \cite[p.~183]{Gar06}.  Since projection bodies are just full-dimensional, $o$-symmetric zonoids, (\ref{pik}) suggests that a characterization of Blaschke addition might follow from a characterization of Minkowski addition as a map between $o$-symmetric zonoids.  In this direction our principal result is as follows.

\smallskip

{\bf Theorem B}. {\em If $n\ge 3$, then an operation between $o$-symmetric zonoids in $\R^n$ is continuous in the Hausdorff metric and $GL(n)$ covariant if and only if $K*L=aK+bL$, for some $a,b\ge 0$ and all $o$-symmetric zonoids $K$ and $L$.}

\smallskip

See Theorem~\ref{characterization Minkowski addition zonoids1}.  A new characterization of Minkowski addition (the case $a=b=1$ of Theorem~B) follows when the identity property is added.  It is interesting to note that the assumptions in Theorem~B are considerably weaker than those in the characterization of Minkowski addition as an operation between $o$-symmetric compact convex sets given in \cite[Corollary~10.5]{GHW}; in particular, we do not need the associativity property in Theorem~B.

It turns out that the L\'{e}vy-Prokhorov metric is exactly the right tool to effect the transition from Theorem~B to Theorem~A.  This requires a certain amount of technical details, but this is apparently unavoidable.  Indeed, we provide a full set of examples showing that none of the properties in our main results can be omitted, nor can uniform continuity in Theorem~A be replaced by continuity.

The paper is organized as follows.  After the preliminary Section~\ref{prelim}, the necessary background for the L\'{e}vy-Prokhorov metrics is set out in Section~\ref{LevProk}.  In Section~\ref{properties} we discuss properties of binary operations.  The new characterization of Minkowski addition is the main result in Section~\ref{Minkowski}, and our main goal, the characterization of Blaschke addition, is achieved in Section~\ref{SecBlaschke}.  Throughout, we use the label ``Proposition" for a result that is known, or suspected to be known.

\section{Definitions and preliminaries}\label{prelim}

As usual, $S^{n-1}$ denotes the unit sphere and $o$ the origin in Euclidean
$n$-space $\R^n$.  We shall assume that $n\ge 2$ throughout.  The unit ball in $\R^n$ will be denoted by $B^n$. The
standard orthonormal basis for $\R^n$ will be $\{e_1,\dots,e_n\}$.  Otherwise, we usually denote the coordinates of $x\in \R^n$ by $x_1,\dots,x_n$.  We write
$[x,y]$ for the line segment with endpoints $x$ and $y$. If $x\in \R^n\setminus\{o\}$, then $x^{\perp}$ is the $(n-1)$-dimensional subspace orthogonal to $x$.  (Throughout the paper, the term {\em subspace} means a linear subspace.)

If $X$ is a set,  we denote by $\partial X$, $\aff X$, $\conv X$, and $\dim X$ the {\it boundary}, {\it affine hull}, {\it convex hull}, and {\it dimension} (that is, the dimension of the affine hull) of $X$, respectively.  If $S$ is a subspace of $\R^n$, then $X|S$ is the (orthogonal) projection of $X$ onto $S$ and $x|S$ is the projection of a vector $x\in\R^n$ onto $S$.

If $t\in\R$, then $tX=\{tx:x\in X\}$. When $t>0$, $tX$ is called a {\em dilatate} of $X$.  The set $-X=(-1)X$ is the {\em reflection} of $X$ in the origin.

A {\it body} is a compact set equal to the closure of its interior.

We write ${\mathcal{H}}^k$ for $k$-dimensional Hausdorff measure in $\R^n$,
where $k\in\{1,\dots, n\}$. The notation $dz$ will always
mean $d{\mathcal{H}}^k(z)$ for the appropriate $k=1,\dots, n$.  If $K$ is a compact convex set in $\R^n$, then $V(K)$ denotes its {\em volume}, that is, ${{\mathcal{H}}^k}(K)$, where $\dim K=k$.  We write $\kappa_n=V(B^n)=\pi^{n/2}/\Gamma(n/2+1)$ for the volume of the unit ball $B^n$.

The Grassmannian of $k$-dimensional subspaces in $\R^n$ is denoted by ${\mathcal{G}}(n,k)$.

A set is {\it $o$-symmetric} if it is centrally symmetric, with center at the
origin.  We shall call a set in $\R^n$ {\em $1$-unconditional} if it is symmetric with respect to each coordinate hyperplane; this is traditional in convex geometry for compact convex sets.

Let ${\mathcal K}^n$ be the class of nonempty compact convex subsets of $\R^n$ and let ${\mathcal{K}}_{o}^n$ be the class of convex bodies in $\R^n$, i.e., members of ${\mathcal K}^n$ with nonempty interiors.  (Note that the same notation is used differently in \cite{GHW}.)  The $o$-symmetric sets in ${\mathcal K}^n$ or ${\mathcal K}_{o}^n$ are denoted by ${\mathcal K}_s^n$ or ${\mathcal K}_{os}^n$, respectively.

If $K$ is a nonempty closed (not necessarily bounded) convex set, then
$$
h_K(x)=\sup\{x\cdot y: y\in K\},
$$
for $x\in\R^n$, is its {\it support function}. A nonempty closed convex set is uniquely determined by its support function.   Support functions are {\em homogeneous of degree 1}, that is,
\begin{equation}\label{homog}
h_K(rx)=rh_K(x),
\end{equation}
for all $x\in \R^n$ and $r\ge 0$, and are therefore often regarded as functions on $S^{n-1}$.  They are also {\em subadditive}, i.e.,
\begin{equation}\label{subadd}
h_K(x+y)\le h_K(x)+h_K(y),
\end{equation}
for all $x,y\in \R^n$.  The {\em Hausdorff distance} $\delta(K,L)$ between sets $K,L\in {\mathcal K}^n$
can be conveniently defined by
\begin{equation}\label{HD}
\delta(K,L)=\|h_K-h_L\|_{\infty},
\end{equation}
where $\|\cdot\|_{\infty}$ denotes the $L_{\infty}$ norm on $S^{n-1}$.  (This is equivalent to the alternative definition
$$\delta(K,L)=\max\{\max_{x\in K}d(x,L),\max_{x\in L}d(x,K)\}$$
that applies to arbitrary compact sets, where $d(x,E)$ denotes the distance
from the point $x$ to the set $E$.)  Proofs of these facts can be found in \cite{Sch93}.  Gruber's book \cite{Gru07} is also a good general reference for convex sets.

The {\em polar set} of an arbitrary set $K$ in $\R^n$ is
$$
K^{\circ}=\{x\in \R^n:\,x\cdot y\le 1 {\mathrm{~for~all~}} y\in K\}.$$
See, for example, \cite[p.~99]{Web}.

The vector or Minkowski sum of sets $X$ and $Y$ in $\R^n$ is defined by
$$
X+Y=\{x+y: x\in X, y\in Y\}.
$$
When $K,L\in {\mathcal K}^n$, $K+L$ can be equivalently defined as the compact convex set such that
$$
h_{K+L}(u)=h_K(u)+h_L(u),
$$
for all $u\in S^{n-1}$.

Let $1< p\le \infty$. Firey \cite{Fir61}, \cite{Fir62} introduced the notion of what is now called the {\em $L_p$ sum} of compact convex sets $K$ and $L$ containing the origin.  (The operation has also been called Firey addition, as in \cite[Section~24.6]{BZ}.)  This is the compact convex set $K+_pL$ defined by
\begin{equation}\label{Lpaddition}
h_{K+_pL}(u)^p=h_K(u)^p+h_L(u)^p,
\end{equation}
for $u\in S^{n-1}$ and $p<\infty$, and by
$$
h_{K+_{\infty}L}(u)=\max\{h_K(u),h_L(u)\},
$$
for all $u\in S^{n-1}$.  In the hands of Lutwak \cite{L1}, \cite{L2}, this definition led to the extensive and powerful $L_p$-Brunn-Minkowski theory (see \cite{GHW} for more information and references). An extension of the $L_p$ sum to arbitrary sets in $\R^n$ was given by  Lutwak, Yang, and Zhang \cite{LYZ} (see also \cite[Example~6.7]{GHW}).

If $K$, $L$, and $M$ are arbitrary sets with $K,L\subset\R^n$ and $M\subset \R^2$, the {\em $M$-sum} of $K$ and $L$ can be defined by
\begin{equation}\label{altdef}
K\oplus_M L= \bigcup\left\{a K+bL : (a,b)\in M\right\}.
\end{equation}
See \cite[Section~6]{GHW}. It appears that $M$-addition was first introduced, for centrally symmetric compact convex sets $K$ and $L$ and a $1$-unconditional compact convex set $M$ in $\R^2$, by Protasov \cite{Pro97}, who proved that $\oplus_M:\left({\mathcal{K}}^n_s\right)^2\rightarrow {\mathcal{K}}^n_s$ for such $M$.  (This proof is omitted in the English translation but can also be found in \cite[Section~6]{GHW}.)

Let ${\mathcal Z}^n$ be the class of {\em zonoids} in $\R^n$, i.e., Hausdorff limits of {\em zonotopes}, finite Minkowski sums of line segments.  Then ${\mathcal Z}_{o}^n$ denotes the sets in ${\mathcal Z}^n$ that have nonempty interior and ${\mathcal Z}_{s}^n$ or ${\mathcal Z}_{os}^n$ signify the $o$-symmetric members of ${\mathcal Z}^n$ or ${\mathcal Z}_{o}^n$, respectively.

The {\em surface area measure} $S(K,\cdot)$ of a convex body $K$ is defined for
Borel subsets $E$ of $S^{n-1}$ by
$$
S(K,E)={\mathcal{H}}^{n-1}\left(g^{-1}(K,E)\right),
$$
where $g^{-1}(K,E)$ is the set of points in $\partial K$ at which there is an outer unit normal vector in $E$.  The quantity $S(K)=S(K,S^{n-1})$ is the {\em surface area} of $K$.  Surface area measures are weakly continuous, meaning that if $K_m, K\in {\mathcal K}^n$ and $\delta(K_m,K)\to 0$ as $m\to \infty$, then $S(K_m,\cdot)$ converges weakly to $S(K,\cdot)$; see \cite[p.~205]{Sch93}.

Let ${\mathcal{M}}_+(S^{n-1})$ be the set of finite nonnegative Borel measures in $S^{n-1}$.  {\em Minkowski's existence theorem} \cite[Theorem~A.3.2]{Gar06}, \cite[Theorem~7.1.2]{Sch93} states that $\mu\in {\mathcal{M}}_+(S^{n-1})$ is the surface area measure of some convex body in $\R^n$ if and only if $\mu$ is not concentrated on any great subsphere of $S^{n-1}$ and its centroid is at the origin.

We define the {\em Blaschke sum} $K\,\sharp\,L$ of convex bodies $K$ and $L$ in $\R^n$ to be the unique convex body with centroid at the origin such that
\begin{equation}\label{Bladef}
S(K\,\sharp \,L,\cdot)=S(K,\cdot)+S(L,\cdot).
\end{equation}
Its existence is guaranteed by Minkowski's existence theorem. Our definition of Blaschke sum agrees with that of Gr\"{u}nbaum \cite{Gru03} (who restricts attention to polytopes).  Schneider's definition \cite[p.~394]{Sch93} has the area centroid \cite[p.~305]{Sch93}, rather than the centroid, at the origin, while in \cite{SchS} the Steiner point is used instead. The position of the Blaschke sum is left open in \cite{Fir65a} and \cite[p.~130]{Gar06}, and in \cite{AKK} it is defined as an operation between translation classes of convex bodies.  With one possible exception (Proposition~\ref{BGLncov}), the results of our paper are unaffected by the choice from these definitions since we usually work with $o$-symmetric sets, for which the centroid, area centroid, and Steiner point all lie at the origin.

The {\it projection body} of $K\in \cK^n$ is the $o$-symmetric set $\Pi K\in
\cK^n_s$ defined by
\begin{equation}\label{Cauchy}
h_{\Pi K}(u)=\frac12\int_{S^{n-1}}|u\cdot v|\,dS(K,v),
\end{equation}
where $u\in S^{n-1}$.  By {\em Cauchy's projection formula} \cite[(A.45), p.~408]{Gar06}, the function on the right of (\ref{Cauchy}) is the {\em brightness function} $b_K$ of $K$, defined by $b_K(u)=V(K|u^{\perp})$, for $u\in S^{n-1}$. Every projection body is an $o$-symmetric zonoid and if $K\in \cK^n_{o}$, then $\Pi K\in \cZ^n_{os}$.  In fact, a set $Z\in \cK^n$ is an $o$-symmetric zonoid if and only if
\begin{equation}\label{zonmeas}
h_{Z}(u)=\frac12\int_{S^{n-1}}|u\cdot v|\,d\mu_Z(v),
\end{equation}
for all $u\in S^{n-1}$, where the uniquely determined {\em even} measure $\mu_Z\in {\mathcal{M}}_+(S^{n-1})$ is called the {\em generating measure} of $Z$.  See \cite[Section~4.1]{Gar06}; the factor 1/2 in (\ref{zonmeas}), usually omitted in the definition of the generating measure, is inserted here for later convenience. Comparing (\ref{Cauchy}) and (\ref{zonmeas}), we see that the generating measure $\mu_{\Pi K}$ of $\Pi K$ is $\left(S(K,\cdot)+S(-K,\cdot)\right)/2$.  Moreover, (\ref{Cauchy}), (\ref{zonmeas}), and Minkowski's existence theorem imply that $\Pi: \cK^n_{os}\to \cZ^n_{os}$ is a bijection.

If $\phi\in GL(n)$ and $K\in \cK^n_{o}$, then by \cite[Theorem~4.1.5]{Gar06},
\begin{equation}\label{Pit}
\Pi(\phi K)=|\det\phi|\phi^{-t}(\Pi K),
\end{equation}
where $\phi^{-t}$ denotes the linear transformation whose standard matrix is the inverse of the transpose of that of $\phi$.  From this it is straightforward to conclude that for $K\in \cZ^n_{os}$,
\begin{equation}\label{Pimt}
\Pi^{-1}(\phi K)=|\det\phi|^{1/(n-1)}\phi^{-t}\left(\Pi^{-1}K\right).
\end{equation}

A {\em mixed volume} $V(K_{i_1},\dots,K_{i_n})$ is a coefficient in the expansion of $V(t_1K_1+\cdots+t_nK_n)$ as a homogeneous polynomial of degree $n$ in the parameters $t_1,\dots,t_n\ge 0$, where $K_1,\dots,K_n\in \cK^n$. The notation $V(K,i;L,n-i)$, for example, means that there are $i$ copies of $K$ and $n-i$ copies of $L$.  Mixed volumes are multilinear (i.e., linear in each variable) and satisfy
\begin{equation}\label{mvc}
V(\phi K_1,\dots,\phi K_n)=|\det\phi|V(K_1,\dots,K_n),
\end{equation}
for $\phi\in GL(n)$. See \cite[Section~A.3]{Gar06} or \cite[Section~5.1]{Sch93}.  We shall also use the formula
\begin{equation}\label{smvol}
V(K;L,n-1)=\int_{S^{n-1}}h_K(u)\,dS(L,u),
\end{equation}
for $K, L\in \cK^n$; see \cite[(A.32), p.~404]{Gar06}.

\section{L\'{e}vy-Prokhorov metrics}\label{LevProk}

Let $\mu, \nu\in {\mathcal{M}}_+(S^{n-1})$ and define
\begin{equation}\label{Pro}
d_{LP}(\mu,\nu)=\inf\{\ee>0:\mu(A)\le\nu(A_{\ee})+\ee,\
\nu(A)\le\mu(A_{\ee})+\ee,\ A \text{ Borel in } S^{n-1}\},
\end{equation}
where
$$A_\ee=\{u\in S^{n-1}:\exists v\in A: |u-v|<\ee\}.$$
Then $d_{LP}$ is a metric on ${\mathcal{M}}_+(S^{n-1})$ called the {\em L\'{e}vy-Prokhorov metric}. It has the property that $d_{LP}(\mu_k,\mu)\to 0$ if and only if $\mu_k$ converges weakly to $\mu$; see \cite[Theorem~6.8]{Bil99}.

By {\em Aleksandrov's uniqueness theorem} \cite[Theorem~3.3.1]{Gar06}, a convex body is uniquely determined, up to translation, by its surface area measure.  It follows that a convex body $K$ with centroid at the origin can be identified with $S(K,\cdot)$. Then the class of such bodies can be given the topology arising from the L\'{e}vy-Prokhorov metric defined by
\begin{equation}\label{LPdef}
\delta_{LP}(K,L)=d_{LP}\left(S(K,\cdot), S(L,\cdot)\right).
\end{equation}
Similarly, the class of $o$-symmetric zonoids can be given the topology arising from the L\'{e}vy-Prokhorov metric defined by
\begin{equation}\label{LPdefZ}
\overline{\delta}_{LP}(K,L)=d_{LP}(\mu_K,\mu_L).
\end{equation}
for $K,L\in \cZ^n_{s}$. Note that the zonoids need not have nonempty interiors.

The following proposition states that the projection body operator $\Pi:{\mathcal K}_{os}^n\to {\mathcal Z}_{os}^n$ is an isometry if we equip each class with their respective L\'{e}vy-Prokhorov metrics.

\begin{prp}\label{projection body operator isometry}
If $K, L\in {\mathcal K}_{os}^n$, then
$$\overline{\delta}_{LP}(\Pi K, \Pi L)=\delta_{LP}(K,L).$$
\end{prp}

\begin{proof}
Since $K, L\in {\mathcal K}_{os}^n$, we have $\mu_{\Pi K}=S(K,\cdot)$ and $\mu_{\Pi L}=S(L,\cdot)$. By (\ref{LPdef}) and (\ref{LPdefZ}), we have
$$\overline{\delta}_{LP}(\Pi K, \Pi L)=d_{LP}(\mu_{\Pi K}, \mu_{\Pi L})=d_{LP}\left(S(K,\cdot), S(L,\cdot)\right)=\delta_{LP}(K,L).$$
\end{proof}

We end this section with two results showing that convergence in either of the L\'{e}vy-Prokhorov metrics $\delta_{LP}$ or $\overline{\delta}_{LP}$ is equivalent to convergence in the Hausdorff metric. The first of these will be used in Examples~\ref{BlasEx1} and~\ref{BlasEx2}, but is not needed for our main result, Theorem~\ref{thmBlaschke}.  It is essentially stated and proved by Fenchel and Jessen \cite[Satz~VIII]{FenJ38}, but is included here for the reader's convenience.

\begin{prp}\label{SecondLPHaus}
Let $K_m,K\in  \cK^n_{os}$ for all $m\in \N$.  Then $\delta_{LP}(K_m,K)\rightarrow 0$ as $m\to \infty$ if and only if $\delta(K_m,K)\rightarrow 0$ as $m\to \infty$.
\end{prp}

\begin{proof}
Suppose that $\delta_{LP}(K_m,K)\rightarrow 0$ as $m\to \infty$.  Then
\begin{equation}\label{skm}
S(K_m)=\int_{S^{n-1}} 1\,d S(K_m,v) \rightarrow \int_{S^{n-1}} 1 \, d S(K,v) = S(K),
\end{equation}
as $m\to \infty$, so there is a $c_0> 0$ such that $S(K_m)\leq c_0$ for all $m \in \N$. By the isoperimetric inequality \cite[(B.14), p.~418]{Gar06}, there is a $c_1>0$ such that $V(K_m) \leq c_1$ for all $m \in \N$.  From (\ref{Cauchy}) and the weak convergence of surface area measures, it follows that for each $u\in S^{n-1}$, $h_{\Pi K_m}(u)\to h_{\Pi K}(u)$ as $m\to \infty$, and hence, by \cite[Theorem~1.8.12]{Sch93}, $h_{\Pi K_m}$ converges uniformly to $h_{\Pi K}$ as $m\to \infty$. Therefore, by (\ref{Cauchy}) and the fact that $\Pi K\in \cK^n_{os}$, there is a $c_2 > 0$ such that
\begin{equation}\label{nconv1}
c_2 \leq \frac 1 2 \int_{S^{n-1}} |u \cdot v| \,d S(K_m,v),
\end{equation}
for all $u \in S^{n-1}$ and $m \in \N$. For any $s > 0$, $u \in S^{n-1}$, and $m \in \N$ with $[-su,su] \subset K_m$, we have
\begin{equation}\label{nconv2}
s|u \cdot v|=h_{[-su,su]}(v) \leq h_{K_m}(v),
\end{equation}
for all $v \in S^{n-1}$. From (\ref{nconv1}) and (\ref{nconv2}), we obtain
$$c_2s
\leq \frac 1 2 \int_{S^{n-1}} s |u \cdot v| \,dS(K_m,v)
\leq \frac 1 2 \int_{S^{n-1}} h_{K_m}(v) \,dS(K_m,v)
= \frac n 2 V(K_m)
\leq c_1n/2. $$
Therefore $s\le R$, where $R=(c_1n)/(2c_2)$.  Using the $o$-symmetry of $K_m$, we conclude that $K_m \subset R B^n$, for all $m \in \N$. By Blaschke's selection theorem \cite[Theorem~1.8.6]{Sch93}, every subsequence of $K_m$ has a subsequence that converges in the Hausdorff metric.  The weak continuity of surface area measures and Aleksandrov's uniqueness theorem \cite[Theorem~3.3.1]{Gar06} force such a subsequence to converge to $K$ in the Hausdorff metric.
Consequently, $\delta(K_m,K)\to 0$ as $m\to \infty$.

Conversely, suppose that $\delta(K_m,K)\to 0$ as $m\to \infty$.  Then $S(K_m,\cdot)$ converges weakly to $S(K,\cdot)$.  By (\ref{LPdef}), it follows that $\delta_{LP}(K_m,K)\rightarrow 0$ as $m\to \infty$.
\end{proof}

The previous proposition holds, more generally, under the assumption that the sets $K$ and $K_m$, $m\in \N$, belong to $\cK^n_{o}$ and have their centroids at the origin.

\begin{lem}\label{zonoids Hausdorff equivalent weak convergence}
Let $Z_m,Z\in  \cZ^n_s$ for all $m\in \N$.  Then $\overline{\delta}_{LP}(Z_m,Z)\rightarrow 0$ as $m\to \infty$ if and only if $\delta(Z_m,Z)\rightarrow 0$ as $m\to \infty$.
\end{lem}

\begin{proof}
Suppose that $\overline{\delta}_{LP}(Z_m,Z)\rightarrow 0$ as $m\to \infty$.  By (\ref{LPdefZ}), $d_{LP}(\mu_{Z_m},\mu_{Z})\rightarrow 0$ as $m\to \infty$, so $\mu_{Z_m}$ converges weakly to $\mu_{Z}$ as $m\to\infty$. In particular, we have
\begin{equation}\label{eq: weak convergence pointwise}
h_{Z_m}(u)=\frac12\int_{\Sphere} |u\cdot v|\,d\mu_{Z_m}(v) \to \frac12\int_{\Sphere} |u\cdot v|\, d\mu_Z(v)=h_{Z}(u),
\end{equation}
as $m\to \infty$, for all $u \in \Sphere$.  By \cite[Theorem~1.8.12]{Sch93},
$h_{Z_m}$ converges to $h_{Z}$ uniformly, so by (\ref{HD}), $\delta(Z_m,Z)\rightarrow 0$ as $m\to \infty$.

Conversely, suppose that $\delta(Z_m,Z)\rightarrow 0$ as $m\to \infty$.  By (\ref{HD}) again, $h_{Z_m}$ converges to $h_{Z}$ uniformly, so \eqref{eq: weak convergence pointwise} holds. By Fubini's theorem and (\ref{zonmeas}),
\begin{align}\label{eq1}
\int_{\Sphere} h_{Z_m}(u) \,du&=\frac12\int_{\Sphere} \int_{\Sphere} |u\cdot v| \,d\mu_{Z_m}(v)\,du\nonumber\\
&=\frac12\int_{\Sphere} \int_{\Sphere} |u\cdot v| \,du\,d\mu_{Z_m}(v) = \kappa_{n-1} \mu_{Z_m}(\Sphere).
\end{align}
Moreover, (\ref{eq1}) holds with $Z_m$ replaced throughout by $Z$.
Since $\delta(Z_m,Z)\rightarrow 0$ as $m\to \infty$, it follows that $\mu_{Z_m}(\Sphere)\le c_3$ for some constant $c_3>0$ and all $m \in \N$ and that
$\mu_Z(\Sphere) \leq c_3$. Now let $f \in C(\Sphere)$ and let $\ee > 0$.
We claim that there is a constant $c_4$ such that
\begin{equation}\label{claim}
\left| \int_{\Sphere} f(v) \,d\mu_{Z_m}(v) - \int_{\Sphere} f(v) \,d\mu_Z(v) \right| < c_4\ee,
\end{equation}
for all sufficiently large $m$. Because generating measures are even, we can assume that $f$ is also even. The span of the set
$$\{f_u(v)=|u\cdot v|,~v \in \Sphere: u \in \Sphere \}$$
is dense in the set of even continuous functions on $\Sphere$; see, for example, the proof of \cite[Theorem~C.2.1]{Gar06} or the references given there. It follows that there are $u_1, \ldots, u_k \in \Sphere$ and $a_1, \ldots, a_k \in \R$ such that if $g(v) =\sum_{i=1}^k a_i |u_i\cdot v|$, for $v\in S^{n-1}$, then
$\|f - g\| < \ee$.  Then
\begin{align*}
\lefteqn{\left| \int_{\Sphere} f(v) \, d\mu_{Z_m}(v) - \int_{\Sphere} f(v) \, d\mu_{Z}(v)\right|}\\
&\leq \left|\int_{\Sphere} f(v) \,d\mu_{Z_m}(v) - \int_{\Sphere} g(v) \,d\mu_{Z_m}(v) \right|+ \left|\int_{\Sphere} g(v) \,d\mu_{Z_m}(v) - \int_{\Sphere} g(v) \, d\mu_{Z}(v)\right| \\
&\phantom{\leq} + \left|\int_{\Sphere} g(v) \,d\mu_{Z}(v) - \int_{\Sphere} f(v) \, d\mu_{Z}(v)\right| \\
&\leq c_3\ee + \ee + c_3\ee,
\end{align*}
for all sufficiently large $m$, where we used the definition of $g$ and (\ref{eq: weak convergence pointwise}) to bound the second term on the right-hand side.  This proves (\ref{claim}) with $c_4=2c_3+1$. It follows that $\mu_{Z_m}$ converges weakly to $\mu_Z$ and hence that $\overline{\delta}_{LP}(Z_m,Z)\rightarrow 0$ as $m\to \infty$.
\end{proof}

\section{Properties of binary operations}\label{properties}

Suppose that $\mathcal{C}\subset\mathcal{D}\subset {\mathcal{K}}^n$. The following properties of operations $*:{\mathcal{C}}^2\rightarrow {\mathcal{D}}$ are supposed to hold for all appropriate $K, L, M, N, K_m, L_m\in{\mathcal{C}}$.

\medskip

1. (Commutativity) $K*L=L*K$.

2. (Associativity) $K*(L*M)=(K*L)*M$.

3.  (Identity) $K*\{o\}=K=\{o\}*K$.

3$'$.  (Limit identity) $\lim_{s\to 0}\left(K*(sB^n)\right)=K=\lim_{s\to 0}\left((sB^n)*K\right)$.

4. (Continuity) $K_m\rightarrow M, L_m\rightarrow N\Rightarrow K_m*L_m\rightarrow M*N$ as $m\rightarrow\infty$.

5. ($GL(n)$ covariance) $\phi(K*L)=\phi K*\phi L$ for all $\phi\in GL(n)$.

6. (Projection covariance) $(K*L)|S=(K|S)*(L|S)$ for all $S\in {\mathcal{G}}(n,k)$, $1\le k\le n-1$.

7. (Monotonicity) $K\subset M$, $L\subset N$ $\Rightarrow K*L\subset M*N$.

\medskip

Of course, the limit identity and continuity properties must be taken with respect to some suitable metric, but in view of Proposition~\ref{SecondLPHaus} and Lemma~\ref{zonoids Hausdorff equivalent weak convergence}, for these properties, any two of the metrics $\delta$, $\delta_{LP}$, and $\overline{\delta}_{LP}$ are interchangeable on the intersection of their domains.  The limit identity property is designed as a substitute for the identity property when $\{o\}\not\in{\mathcal{C}}$, for example, when ${\mathcal{C}}={\mathcal{K}}_{os}^n$. In the definition of projection covariance, the stated property is to hold for all $1\le k\le n-1$.  However, our results never require $k>2$.

It follows from (\ref{Bladef}) that Blaschke addition is commutative and associative.  When $n=2$, Blaschke addition is the same as Minkowski addition, up to translation, so it can be extended to an operation between $o$-symmetric compact convex sets in $\R^2$ which is continuous in the Hausdorff metric.  For $n\ge 3$, such an extension does not exist, by \cite[Theorem~5.3]{GHW}.  Neither the identity property nor projection covariance apply to Blaschke addition, which is only defined for convex bodies.  From (\ref{Bladef}) it follows easily that Blaschke addition has the limit identity property with respect to the metric $\delta_{LP}$, and hence by Proposition~\ref{SecondLPHaus}, also with respect to the Hausdorff metric, at least for $o$-symmetric sets.

A proof of the following proposition was sketched by Firey \cite[p.~34]{Fir65}.  We include a detailed proof for the reader's convenience.

\begin{prp}\label{BGLncov}
Blaschke addition $\sharp: \left({\mathcal{K}}_{o}^n\right)^2\to {\mathcal{K}}_{o}^n$ is $GL(n)$ covariant.
\end{prp}

\begin{proof}
Let $\phi\in GL(n)$ and let $K,L\in {\mathcal{K}}_{o}^n$.  If $M$ is a nonempty compact convex set in $\R^n$, then using the definition of Blaschke addition, (\ref{smvol}), the multilinearity of mixed volumes, and (\ref{mvc}), we obtain
\begin{eqnarray*}
\int_{S^{n-1}}h_M(u)\,dS(\phi K\,\sharp\, \phi L,u)
&=&\int_{S^{n-1}}h_M(u)\, dS(\phi K ,u)+\int_{S^{n-1}}h_M(u)\,
dS(\phi L,u)\\
&=&V(M;\phi K,n-1)+V(M;\phi L, n-1)\\
&=&|\det\phi| \left(V\left(\phi^{-1}M;K, n-1\right)+V\left(\phi^{-1}M;L,n-1\right)\right)\\
&=&|\det\phi|\int_{S^{n-1}}h_{\phi^{-1} M}(u)\,dS(K\,\sharp\,  L,u)\\
&=&|\det\phi|V\left(\phi^{-1}M;K\,\sharp\, L,n-1\right)\\
&=&V(M;\phi(K\,\sharp\, L),n-1)\\
&=&\int_{S^{n-1}}h_M(u)\,dS(\phi (K\,\sharp\, L),u).
\end{eqnarray*}
The resulting equation therefore holds when $h_M$ is replaced by a difference of support functions of compact convex sets.  Since the latter are dense in the space $C(S^{n-1})$ of continuous functions on $S^{n-1}$ (see \cite[Lemma~1.7.9]{Sch93}), the equation also holds when $h_M$ is replaced by any $f\in C(S^{n-1})$.  It follows that $S(\phi K\,\sharp\, \phi L,\cdot)=S(\phi (K\,\sharp\, L),\cdot)$ and hence, by Aleksandrov's uniqueness theorem \cite[Theorem~3.3.1]{Gar06}, that $\phi (K\,\sharp\, L)$ is a translate of $\phi K\,\sharp\, \phi L$.  By our definition of Blaschke addition, the centroids of $\phi K\,\sharp\, \phi L$ and $K\,\sharp\,L$ are at the origin.  Using the latter, the definition of centroid, and the linearity of integrals, we see that the centroid of $\phi(K\,\sharp\,L)$ is also at the origin. Therefore
$$\phi (K\,\sharp\, L)=\phi K\,\sharp\, \phi L,$$
as required.
\end{proof}

We take this opportunity to note that, somewhat surprisingly, Blaschke addition is not monotonic, no matter which of the standard definitions is used.

\begin{ex}\label{NotMonotone}
{\em Let $K=[-1/2,1/2]^3\subset\R^3$ and let $L=\phi K$, where $\phi$ is a rotation by $\pi/4$ around the $x_3$-axis.  Then $K$ and $L$ are $o$-symmetric cubes.  Let $M=\conv\{K, L\}$, so that $M$ is an $o$-symmetric cylinder of height $1$ with the $x_3$-axis as axis and a regular octagon as base.  Obviously $K,L\subset M$. By adding the surface area measures of $K$ and $L$, it is easy to see that $K\,\sharp\,L$ is also an $o$-symmetric cylinder with the $x_3$-axis as axis and a regular octagon as base, whose vertical facets each have area $1$ and whose horizontal facets both have area 2.  Let $h$ be the height of $K\,\sharp\,L$ and let $s$ be the length of one of its horizontal edges.  Then we have $hs=1$ and
$$2s^2\cot(\pi/8)=2(1+\sqrt{2})s^2=2.$$
From these two equations, we obtain $h=\sqrt{1+\sqrt{2}}>\sqrt{2}$.  However, $M\,\sharp\,M=\sqrt{2}M$ has height $\sqrt{2}$, so no translate of $K\,\sharp\,L$ is contained in $M\,\sharp\,M$.}
\end{ex}

Much more information about properties of known operations between compact convex sets can be found in \cite{GHW}.

Recall that we assume that $n\ge 2$ throughout this paper.

\begin{prp}\label{FromGHW1}
An operation $*:\left({\mathcal{Z}}^n_s\right)^2\rightarrow {\mathcal{K}}^n$ is projection covariant if and only if it can be defined by
\begin{equation}\label{nn}
h_{K*L}(x)=h_{M}\left(h_K(x),h_L(x)\right),
\end{equation}
for all $K,L\in {\mathcal{Z}}^n_s$ and $x\in\R^n$, or equivalently by
\begin{equation}\label{Meq}
K*L=K\oplus_M L,
\end{equation}
where $M$ is a $1$-unconditional compact convex set in $\R^2$.  Moreover, $M$ is uniquely determined by $*$.
\end{prp}

\begin{proof}
The result was proved in \cite[Theorem~7.6]{GHW} for operations $*:\left({\mathcal{K}}^n_s\right)^2\rightarrow {\mathcal{K}}^n$.  However the same proof works for operations $*:\left({\mathcal{Z}}^n_s\right)^2\rightarrow {\mathcal{K}}^n$.  To see this, note that in the proof of the auxiliary \cite[Lemma~7.4]{GHW} only $o$-symmetric line segments are used, and the proof of \cite[Theorem~7.6]{GHW} employs only $o$-symmetric line segments and $o$-symmetric balls.
\end{proof}

\begin{prp}\label{FromGHW2}
An operation $*:\left({\mathcal{Z}}^n_s\right)^2\rightarrow {\mathcal{K}}^n$  is projection covariant if and only if it is continuous in the Hausdorff metric and $GL(n)$ covariant.
\end{prp}

\begin{proof}
If $*$ is continuous and $GL(n)$ covariant, then it is projection covariant by \cite[Lemma~4.1]{GHW}. Since $\oplus_M:\left({\mathcal{K}}^n_s\right)^2\rightarrow {\mathcal{K}}_s^n$ is continuous in the Hausdorff metric and $GL(n)$ covariant, the converse follows from Proposition~\ref{FromGHW1}.
\end{proof}

\begin{prp}\label{FromGHW3}
An operation $*:\left({\mathcal{Z}}^n_s\right)^2\rightarrow {\mathcal{K}}^n$ is projection covariant and associative if and only if $*=\oplus_M$, where either $M=\{o\}$, or $M=[-e_1,e_1]$, or $M=[-e_2,e_2]$, or $M$ is the unit ball in $l^2_p$ for some $1\le p\le \infty$; in other words, if and only if either $K*L=\{o\}$, or $K*L=K$, or $K*L=L$, for all $K,L\in {\mathcal{Z}}^n_s$, or else $*=+_p$ for some $1\le p\le\infty$.
\end{prp}

\begin{proof}
Again, the result was proved in \cite[Theorem~7.9]{GHW} for operations $*:\left({\mathcal{K}}^n_s\right)^2\rightarrow {\mathcal{K}}^n$, but since the proof only uses $o$-symmetric balls, it also applies to operations $*:\left({\mathcal{Z}}^n_s\right)^2\rightarrow {\mathcal{K}}^n$.
\end{proof}

\section{A characterization of Minkowski addition}\label{Minkowski}

\begin{thm}\label{characterization Minkowski addition zonoids1}
If $n\ge 3$, then $*:\left( \cZ^n_s\right)^2\to \cZ^n_s$ is projection covariant (or, equivalently, continuous in the Hausdorff metric and $GL(n)$ covariant) if and only if $K*L=aK+bL$, for some $a,b\ge 0$ and all $K,L\in {\mathcal{Z}}^n_s$.
\end{thm}

\begin{proof}
By Proposition~\ref{FromGHW1}, $*=\oplus_M$ for some $1$-unconditional convex body $M$ in $\R^2$, so for this $M$, the equation (\ref{nn}) holds
for all $K,L\in {\mathcal{Z}}^n_s$ and $x\in\R^n$.  We have to show that $M = [-a,a] \times [-b,b]$, as this is the unique $1$-unconditional convex body $M$ for which (\ref{nn}) is equivalent to
$$h_{K*L}(x)=ah_K(x)+bh_L(x)=h_{aK+bL},$$
for all $K,L\in {\mathcal{Z}}^n_s$ and $x\in\R^n$.  It will suffice to show that
\begin{equation}\label{goal}
    h_M(1,1) = h_M(1,0) + h_M(0,1),
\end{equation}
since if $a=h_M(1,0)$ and $b=h_M(0,1)$, then $v=(a,b)$ is the
only point in $\R^2$ satisfying $v \cdot (1,0) \leq
h_M(1,0)$, $v \cdot (0,1) \leq h_M(0,1)$, and $v \cdot (1,1) =
h_M(1,1)$. Therefore $(a,b) \in M$ and the result follows from
the fact that $M$ is $1$-unconditional.

Let $n\ge 3$, let $K=\conv\{\pm e_1,\pm e_2\}$, and let $L = [-e_3, e_3]$.  By assumption, $K*L$ is a zonoid, so its support function must satisfy Hlawka's inequality
\begin{equation}\label{eq: Hlawka's inequality}
h_{K*L}(x) + h_{K*L}(y) + h_{K*L}(z) + h_{K*L}(x+y+z) \geq h_{K*L}(x+y) + h_{K*L}(x+z) + h_{K*L}(y+z),
\end{equation}
for all $x,y,z \in \R^n$; see \cite[Theorem~3.4]{GooW93}.

If $w=(w_1,\dots,w_n)\in \R^n$, then $h_K(w) = \max\{|w_1|,|w_2|\}$ and $h_L(w) = |w_3|$.  For $s>0$, let $x = (-1,1,s,0,\dots,0)$, $y = (1,-1,0,\dots,0)$, and $z = (1,1,0,\dots,0)$.  Then
$x+y+z =(1,1,s,0,\dots,0)$, $x+y = (0,0,s,0,\dots,0)$, $x+z = (0,2,s,0,\dots,0)$, and $y+z = (2,0,\dots,0)$.  Substituting into (\ref{eq: Hlawka's inequality}) and using (\ref{nn}), we obtain
$$
h_M(1,s) + h_M(1,0) + h_M(1,0) + h_M(1,s)\geq h_M(0,s) + h_M(2,s) + h_M(2,0).$$
In view of the homogeneity (\ref{homog}) of $h_M$, this reduces to
$$
h_M(2,2s) \geq h_M(0,s) + h_M(2,s).
$$
By the subadditivity (\ref{subadd}) of $h_M$, the reverse of the previous inequality holds, so
\begin{equation}\label{f1}
h_M(2,2s) = h_M(0,s) + h_M(2,s),
\end{equation}
for all $s>0$.  Taking instead $K=[-e_1,e_1]$, $L=\conv\{\pm e_2,\pm e_3\}$, $x = (t,-1,1,0,\dots,0)$, $y = (0,1,-1,0,\dots,0)$, and $z = (0,1,1,0,\dots,0)$, similar computations yield
\begin{equation}\label{f2}
h_M(2t,2) = h_M(t,0) + h_M(t,2),
\end{equation}
for all $t>0$.
Setting $s = 1$ in (\ref{f1}) and $t = 2$ in (\ref{f2}), and using (\ref{homog}) again, we arrive at
$$
2h_M(1,1) = h_M(0,1) + h_M(2,1) \quad \text{and} \quad h_M(2,1) = h_M(1,0) + h_M(1,1).
$$
Substituting for $h_M(2,1)$ from the second of these equations into the first, we obtain (\ref{goal}).
\end{proof}

The following characterization of Minkowski addition follows immediately from Theorem~\ref{characterization Minkowski addition zonoids1} and the definition of the identity property (Property~3 in Section~\ref{properties}).  By Proposition~\ref{FromGHW2}, the first two properties listed can be replaced by projection covariance.

\begin{cor}\label{characterization Minkowski}
If $n\ge 3$, then $*:\left( \cZ^n_s\right)^2\to \cZ^n_s$ is continuous in the Hausdorff metric, $GL(n)$ covariant, and has the identity property if and only if it is Minkowski addition.
\end{cor}

We now provide three examples that show that none of the assumptions on the operation $*$ in Theorem~\ref{characterization Minkowski addition zonoids1} and Corollary~\ref{characterization Minkowski} can be omitted.  Where possible, we produce operations that are also associative.

\begin{ex}\label{MinkEx1}
{\em In \cite[Example~7.15]{GHW}, an operation $*:\left( \cK^n_s\right)^2\to \cK^n_s$ is defined as follows.  Let $F:{\mathcal{K}}^n_s\to {\mathcal{K}}^n_s$ be such that $F(K)$ is the set obtained by rotating $K$ by an angle equal to its volume $V(K)$ around the origin in the $\{x_1,x_2\}$-plane and define
\begin{equation}\label{eqMEx1}
K*L=F^{-1}\left(F(K)+F(L)\right),
\end{equation}
for all $K,L\in {\mathcal{K}}^n_s$.  Clearly $*:\left( \cZ^n_s\right)^2\to\cZ^n_s$.  It is easy to see that $*$ is continuous in the Hausdorff metric and has the identity property (and is also associative) but in \cite[Example~7.15]{GHW} it is shown that it is neither projection covariant nor $GL(n)$ covariant.  When $n\ge 3$, this can also be seen directly from Theorem~\ref{characterization Minkowski addition zonoids1}, as follows.

Let $a,b\ge 0$ and let $K=L=[-1/2,1/2]^n$, so that $V(K)=V(L)=1$ and $aK+bL=[-(a+b)/2,(a+b)/2]^n$. Then $F(K)$ and $F(L)$ are rotations of $[-1/2,1/2]^n$ by an angle of $1$ around the origin in the $\{x_1,x_2\}$-plane, so $F(K)+F(L)$ is a rotation of $[-1,1]^n$ by an angle of $1$ around the origin in the $\{x_1,x_2\}$-plane.  Since $V\left(F(K)+F(L)\right)=2^n$, $K*L$ is a rotation of $[-1,1]^n$ by an angle of $1-2^n$ around the origin in the $\{x_1,x_2\}$-plane.  This shows that $K*L\neq aK+bL$.}
\end{ex}

\begin{ex}\label{MinkEx2}
{\em Define $*$ on $\left( \cZ^n_s\right)^2$ by $K*L=JK+_2JL$, where $JK$ is the John ellipsoid of $K$ taken in $\aff K$ and $+_2$ is defined by (\ref{Lpaddition}) with $p=2$. By definition, the John ellipsoid is the ellipsoid of maximal volume contained in $K$.  The existence of $JK$, and the fact that it is $o$-symmetric whenever $K$ is, is proved in \cite[Theorem~4.2.12]{Gar06}, for example. Since $JK$ and $JL$ are $o$-symmetric, (\ref{Lpaddition}) with $p=2$ shows that $JK+_2JL$ is also $o$-symmetric.  Moreover, the fact that $JK$ and $JL$ are ellipsoids in $\R^n$ ensures that $JK+_2JL$ is also an ellipsoid.  Indeed, Firey \cite{Fir64} proved this for $n$-dimensional ellipsoids.  If $JK$ or $JL$ are not $n$-dimensional, choose sequences of $n$-dimensional ellipsoids $E_m$ and $F_m$ such that $E_m\to JK$ and $F_m\to JL$ as $m\to \infty$ in the Hausdorff metric.  Then
$E_m+_2F_m\to JK+_2JL$ as $m\to\infty$, since $+_2$ is continuous in the Hausdorff metric as an operation between compact convex sets containing the origin.  Then $JK+_2JL$ is the limit in the Hausdorff metric of $n$-dimensional ellipsoids and so must itself be an ellipsoid.  By \cite[Corollary~4.1.6 and Theorem~4.1.11]{Gar06}, every $n$-dimensional ellipsoid is a zonoid and it follows by taking limits in the Hausdorff metric that every ellipsoid is a zonoid.
Consequently, $JK+_2JL$ is a zonoid and hence $*:\left( \cZ^n_s\right)^2\to \cZ^n_s$.

Let $\phi\in GL(n)$.  The formula $\phi(JK)=J(\phi K)$ is noted in \cite[Lemma~2.5]{LYZ05} (taking $p=\infty$ there, which corresponds to the John ellipsoid) when $K\in \cK^n_{os}$.
If $K\in \cK^n_s$ and $\dim K<n$, then this formula is clearly
true for those $\phi$ such that $\phi(\aff K)=\aff K$ and for orthogonal $\phi$. Since a general $\phi\in GL(n)$ is a composition of such maps, the formula holds for all $K\in \cK^n_s$.  Using this and the fact that $+_2$ is $GL(n)$ covariant as an operation between compact convex sets containing the origin, we have
$$\phi(K*L)=\phi(JK+_2JL)=\phi(JK)+_2\phi(JL)=J(\phi K)+J(\phi L)=\phi K*\phi L.$$
This shows that $*$ is $GL(n)$ covariant and by Theorem~\ref{characterization Minkowski addition zonoids1}, when $n\ge 3$ it is therefore neither projection covariant nor continuous in the Hausdorff metric.

The operation $*$ is also associative, because if $K$ is an ellipsoid, then clearly $JK=K$.  Therefore if $K,L,M\in \cZ^n_s$, the associativity of the operation $+_2$ yields
$$K*(L*M)=JK+_2J(JL+_2JM)=JK+_2(JL+_2JM)=(JK+_2JL)+_2JM=(K*L)*M.$$}
\end{ex}

\begin{ex}\label{MinkEx3}
{\em The operation $*:\left( \cZ^n_s\right)^2\to\cZ^n_s$ defined by $K*L=K+2L$ for $K,L\in \cZ^n_s$ is projection covariant (and hence continuous in the Hausdorff metric and $GL(n)$ covariant) but does not satisfy the identity property.}
\end{ex}

\section{A characterization of Blaschke addition}\label{SecBlaschke}

\begin{lem}\label{lemBlaschke}
As an operation $\sharp \colon \left(\cK^n_{os}\right)^2  \to \cK^n_{os}$, Blaschke addition is uniformly continuous in the L\'{e}vy-Prokhorov metric $\delta_{LP}$.
\end{lem}

\begin{proof}
Let $\mu_i,\nu_i\in {\mathcal{M}}_+(S^{n-1})$, $i=1,2$, and suppose that
$$\max\{d_{LP}(\mu_1,\mu_2),d_{LP}(\nu_1,\nu_2)\}<\ee.$$
Then for each Borel set $A$ in $S^{n-1}$, we have
\begin{align*}
(\mu_1 + \nu_1)(A)
&= \mu_1(A) + \nu_1(A) \\
&\leq \mu_2(A^\ee) + \ee + \nu_2(A^\ee) + \ee \\
&= (\mu_2 + \nu_2)(A^\ee) + 2 \ee\leq (\mu_2 + \nu_2)(A^{2 \ee}) + 2 \ee .
\end{align*}
It follows that
\begin{equation}\label{dlp}
d_{LP}(\mu_1+\nu_1,\mu_2+\nu_2 ) \leq 2\max\{d_{LP}(\mu_1,\mu_2),d_{LP}(\nu_1,\nu_2)\}.
\end{equation}
If $K_i,L_i\in \cK^n_{os}$, $i=1,2$, we can apply (\ref{Bladef}), (\ref{dlp}) with $\mu_i=S(K_i,\cdot)$ and $\nu_i=S(L_i,\cdot)$, $i=1,2$, and (\ref{LPdef}) to obtain
\begin{equation}\label{esti}
\delta_{LP}(K_1\,\sharp\, L_1,K_2\,\sharp\, L_2) \leq 2\max\{\delta_{LP}(K_1,K_2),\delta_{LP}(L_1,L_2)\},
\end{equation}
so $\sharp \colon \left(\cK^n_{os}\right)^2  \to \cK^n_{os}$ is uniformly continuous with respect to $\delta_{LP}$.
\end{proof}

\begin{thm}\label{thmBlaschke}
If $n\ge 3$, then $*:\left( \cK^n_{os}\right)^2\to \cK^n_{os}$ is uniformly continuous in the L\'{e}vy-Prokhorov metric $\delta_{LP}$ and $GL(n)$ covariant if and only if for all $K,L\in \cK^n_{os}$, we have either $K*L=aK$, for some $a>0$, or $K*L=bL$, for some $b>0$, or $K*L=aK\,\sharp\,bL$, for some $a,b>0$.
\end{thm}

\begin{proof}
Let $n\ge 3$ and let $*:\left( \cK^n_{os}\right)^2\to \cK^n_{os}$ have the properties stated in the theorem.  We define a new operation $\diamond \colon \left(\cZ^n_{os}\right)^2 \to \cZ^n_{os}$ by
\begin{equation}\label{didef}
K \diamond L = \Pi \left( \Pi^{-1} K *\Pi^{-1} L \right),
\end{equation}
for all $K,L \in \cZ^n_{os}$.  We shall prove that $\diamond$ also has the properties stated in the theorem.

Let $\ee>0$. The uniform continuity of $*$ implies that there exists a $\delta>0$ such that
\begin{equation}\label{eqee}
\delta_{LP}\left( K_1*L_1, K_2*L_2 \right) < \ee,
\end{equation}
for all $K_i,L_i \in \cK^n_{os}$, $i=1,2$, such that $\delta_{LP}(K_1,K_2)< \delta$ and $\delta_{LP}(L_1,L_2)<\delta$.  Let $Y_i,Z_i\in \cZ^n_{os}$ satisfy
$\overline{\delta}_{LP}(Y_1,Y_2) < \delta$ and $\overline{\delta}_{LP}(Z_1,Z_2) < \delta$.
Then $\Pi^{-1}Y_i,\Pi^{-1}Z_i\in \cK^n_{os}$, $i=1,2$, and by Proposition~\ref{projection body operator isometry},
$$\delta_{LP}\left( \Pi^{-1}Y_1, \Pi^{-1}Y_2 \right) = \overline{\delta}_{LP}(Y_1,Y_2) < \delta~~{\text{and}}~~\delta_{LP}\left( \Pi^{-1}Z_1, \Pi^{-1}Z_2 \right) = \overline{\delta}_{LP}(Z_1,Z_2) < \delta.$$
Using (\ref{didef}), Proposition~\ref{projection body operator isometry} again, and (\ref{eqee}) with $K_i=\Pi^{-1}Y_i$ and $L_i=\Pi^{-1}Z_i$, $i=1,2$, we obtain
\begin{align*}
\overline{\delta}_{LP}(Y_1 \diamond Z_1, Y_2 \diamond Z_2)
&=\overline{\delta}_{LP}\left( \Pi\left(\Pi^{-1}Y_1 *\Pi^{-1}Z_1\right), \Pi\left(\Pi^{-1}Y_2*\Pi^{-1}Z_2\right) \right)\\
&= \delta_{LP}\left( \Pi^{-1}Y_1 *\Pi^{-1}Z_1, \Pi^{-1}Y_2*\Pi^{-1}Z_2 \right)< \ee.
\end{align*}
Therefore $\diamond$ is uniformly continuous in the L\'{e}vy-Prokhorov metric $\overline{\delta}_{LP}$.

Let $K,L \in \cZ^n_{os}$ and let $\phi \in GL(n)$.  Using (\ref{didef}), (\ref{Pimt}), the $GL(n)$ covariance of $*$, and (\ref{Pit}), we obtain
\begin{align*}
\phi K \diamond \phi L
&= \Pi \left( \Pi^{-1} (\phi K)*\Pi^{-1} (\phi L) \right) \\
&= \Pi \left( \left(|\det\phi|^{1/(n-1)}\phi^{-t} \left(\Pi^{-1} K\right) \right)*\left( |\det\phi|^{1/(n-1)}\phi^{-t} \left(\Pi^{-1} L\right) \right) \right) \\
&= \Pi \left( |\det\phi|^{1/(n-1)}\phi^{-t} \left( \Pi^{-1} K*\Pi^{-1} L \right) \right) \\
&= \phi\left( \Pi \left( \Pi^{-1} K*\Pi^{-1} L \right)\right)=\phi (K \diamond L),
\end{align*}
proving that $\diamond$ is $GL(n)$-covariant.

The set $\cZ^n_{os}$ is dense in $\left(\cZ^n_s,\delta\right)$, so
by Lemma~\ref{zonoids Hausdorff equivalent weak convergence}, it is also dense in $\left(\cZ^n_s, \overline{\delta}_{LP}\right)$.  We claim that $\left(\cZ^n_s, \overline{\delta}_{LP}\right)$ is a complete metric space.  Indeed, $\left({\mathcal{M}}_+(S^{n-1}),d_{LP}\right)$ is complete, because $S^{n-1}$ is separable and complete (see \cite[Theorem~6.8]{Bil99}). Using the fact that a measure is even if and only if the integral of every continuous odd function with respect to the measure is zero, it is easy to see that the set of even measures in $\left({\mathcal{M}}_+(S^{n-1}),d_{LP}\right)$ is closed and therefore also complete.  Since the set of generating measures of zonoids is precisely this set of even measures, it too is complete.  It then follows from (\ref{LPdefZ}) that $\left(\cZ^n_s, \overline{\delta}_{LP}\right)$ is complete, proving the claim.

Since $\diamond$ is uniformly continuous with respect to $\overline{\delta}_{LP}$, the properties of $\cZ^n_{os}$ and $\cZ^n_s$ established in the previous paragraph ensure that it has a unique extension to a map $\diamond\colon \left(\cZ^n_s\right)^2\to\cZ^n_s$ that is uniformly continuous with respect to $\overline{\delta}_{LP}$; see, for example, \cite[Theorem~D, p.~78]{Sim03}. By Lemma~\ref{zonoids Hausdorff equivalent weak convergence} again, this extension is also continuous in the Hausdorff metric.

Using the continuity of $\diamond\colon \left(\cZ^n_s\right)^2\to\cZ^n_s$ in the Hausdorff metric, it is easy to check that the $GL(n)$ covariance of $\diamond$, proved above for $\cZ^n_{os}$, also holds for $\cZ^n_s$.  It now follows from Theorem~\ref{characterization Minkowski addition zonoids1}
that $K\diamond L=cK+dL$, for some $c,d\ge 0$ and all $K,L\in {\mathcal{Z}}^n_s$.

Let $a=c^{1/(n-1)}$ and $b=d^{1/(n-1)}$ and let $K,L\in \cK^n_{os}$.  Then $a,b\ge 0$, and $Y=\Pi K$ and $Z=\Pi L$ belong to $\cZ^n_{os}$.  From (\ref{Pit}) it is easy to see that $\Pi(rK)=r^{n-1}\Pi K$, for all $r\ge 0$. Using this and (\ref{didef}) with $K$ and $L$ replaced by $Y$ and $Z$, respectively, we obtain, for $c,d>0$ (and hence $a,b>0$),
\begin{align*}
\Pi(aK\,\sharp\, bL)&=\Pi(aK)+\Pi(bL)=a^{n-1}\Pi K+b^{n-1}\Pi L\\
&=
cY+dZ=Y\diamond Z=\Pi \left( \Pi^{-1}Y *\Pi^{-1}Z \right)=\Pi(K*L).
\end{align*}
Since $aK\,\sharp\, bL$ and $K*L$ are $o$-symmetric, we conclude that $K*L=aK\,\sharp\, bL$.  If $c>0$ and $d=0$, we have $a>0$ and $Y\diamond Z=cY$.  Therefore
$$\Pi(aK)=a^{n-1}\Pi K=cY=Y\diamond Z=\Pi \left( \Pi^{-1}Y *\Pi^{-1}Z \right)=\Pi(K*L),$$
from which we obtain $K*L=aK$.  The case when $a=0$ and $b>0$ is similar.  If $c=d=0$, then
$$\{o\}=cY+dZ=Y\diamond Z=\Pi \left( \Pi^{-1}Y *\Pi^{-1}Z \right)=\Pi(K*L),$$
which is impossible because $K*L\in \cK^n_{os}$ by assumption.

Finally, $\sharp:\left(\cK^n_{os}\right)^2\to \cK^n_{os}$ is $GL(n)$ covariant by Proposition~\ref{BGLncov} and Lemma~\ref{lemBlaschke} shows that it is uniformly continuous in the L\'{e}vy-Prokhorov metric $\delta_{LP}$.  Therefore the map $*:\left( \cK^n_{os}\right)^2\to \cK^n_{os}$ defined by $K*L=aK\,\sharp\,bL$, for some $a,b> 0$, also has these properties.
\end{proof}

The following characterization of Blaschke addition follows immediately from Theorem~\ref{thmBlaschke} and the definition of the limit identity property (Property~3$'$ in Section~\ref{properties}), which we assume is taken with respect to the L\'{e}vy-Prokhorov metric $\delta_{LP}$ (or equivalently, by Proposition~\ref{SecondLPHaus}, with respect to the Hausdorff metric).

\begin{cor}\label{characterization Blaschke}
If $n\ge 3$, then $*:\left( \cK^n_{os}\right)^2\to \cK^n_{os}$ is uniformly continuous in the L\'{e}vy-Prokhorov metric $\delta_{LP}$, $GL(n)$ covariant, and has the limit identity property,  if and only if it is Blaschke addition.
\end{cor}

We now provide examples that show that none of the assumptions on the operation $*$ in Theorem~\ref{thmBlaschke} and Corollary~\ref{characterization Blaschke} can be omitted and moreover that uniform continuity cannot be replaced by continuity.  Where possible, we exhibit operations that are also associative.

\begin{ex}\label{BlasEx1}
{\em Minkowski addition $+:\left( \cK^n_{os}\right)^2\to \cK^n_{os}$ is $GL(n)$ covariant and therefore, by Theorem~\ref{thmBlaschke}, not uniformly continuous in the L\'{e}vy-Prokhorov metric ${\delta}_{LP}$ when $n\ge 3$.  It also has the limit identity property because as an operation $+:\left( \cK^n_{s}\right)^2\to \cK^n_{s}$, it has the identity property and is continuous in the Hausdorff metric.   Since $+:\left( \cK^n_{os}\right)^2\to \cK^n_{os}$ is continuous in the Hausdorff metric, Proposition~\ref{SecondLPHaus} implies that it is also continuous in the L\'{e}vy-Prokhorov metric ${\delta}_{LP}$.}
\end{ex}

Another operation with the same properties as those of Minkowski addition given in Example~\ref{BlasEx1} was introduced by Firey \cite{Fir61}.  In \cite[Section~5.5]{GHW}, it is called polar $L_p$ addition and is defined for $1\le p\le \infty$ and $K,L\in {\mathcal K}^n_{o}$ by $\left(K^{\circ}+_{p} L^{\circ}\right)^{\circ}$.

The following example is inspired by \cite[Example~7.15]{GHW} (see Example~\ref{MinkEx1}).

\begin{ex}\label{BlasEx2}
{\em Let $n\ge 2$ and let $F:{\mathcal{K}}^n_{os}\to {\mathcal{K}}^n_{os}$ be such that $F(K)$ is the set obtained by rotating $K$ by an angle equal to its surface area $S(K)$ around the origin in the $\{x_1,x_2\}$-plane.  Note that since $S\left(F(K)\right)=S(K)$, the map $F$ is injective and so $F^{-1}$ is defined.  Of course, $F^{-1}$ rotates by an angle $-S(K)$ instead. Now define
\begin{equation}\label{eqBEx3}
K*L=F^{-1}\left(F(K)\,\sharp\,F(L)\right),
\end{equation}
for all $K,L\in {\mathcal{K}}^n_{os}$.  It is easy to check that $*$ is associative.

We claim that $*$ is uniformly continuous in the L\'{e}vy-Prokhorov metric ${\delta}_{LP}$.  To see this, suppose that $K,L\in {\mathcal{K}}^n_{os}$ are such that
\begin{equation}\label{e1}
\delta_{LP}(K,L)=d_{LP}\left(S(K,\cdot),S(L,\cdot)\right)=\delta>0.
\end{equation}
By (\ref{Pro}) with $A=S^{n-1}$, this implies that
\begin{equation}\label{e2}
|S(K)-S(L)|=|S(K,S^{n-1})-S(L,S^{n-1})|<2\delta.
\end{equation}
Let $F(K)=\theta_KK$ and $F_L=\theta_LL$, where $\theta_K$ and $\theta_L$ are the appropriate rotations of $K$ and $L$, respectively.  Let $A$ be a Borel subset of $S^{n-1}$.  From (\ref{e2}) and the definition of $F$ we conclude that the angles of rotation of $\theta_K$ and $\theta_L$ differ by less than $2\delta$, so $\theta_K^{-1}A\subset \left(\theta_L^{-1}A\right)_{2\delta}$.  Using this and (\ref{e1}), we obtain
\begin{align*}
S(\theta_KK, A)&=S(K,\theta_K^{-1}A)\\
&<
S\left(L,\left(\theta_K^{-1}A\right)_{2\delta}\right)+2\delta\\
&<
S\left(L,\left(\theta_L^{-1}A\right)_{4\delta}\right)+4\delta\\
&=
S\left(L,\theta_L^{-1}A_{4\delta}\right)+4\delta=
S\left(\theta_LL,A_{4\delta}\right)+4\delta.
\end{align*}
Since $K$ and $L$ can be interchanged in this argument, this proves that
\begin{equation}\label{e3}
\delta_{LP}\left(F(K),F(L)\right)
=d_{LP}\left(S\left(F(K),\cdot\right),S\left(F(L),\cdot\right)\right)
\le 4\delta=4\delta_{LP}(K,L).
\end{equation}

Now let $\ee>0$ and suppose that $K_i,L_i\in \cK^n_{os}$, $i=1,2$, satisfy
\begin{equation}\label{e4}
\delta_{LP}(K_1,K_2)<\ee/32~~\quad{\text{and}}~~\quad\delta_{LP}(L_1,L_2)<\ee/32.
\end{equation}
Using (\ref{e3}) with $F$ replaced by $F^{-1}$, (\ref{esti}), (\ref{e3}) again, and (\ref{e4}), we obtain
\begin{align*}
\delta_{LP}(K_1*L_1,K_2*L_2)&=
\delta_{LP}\left(F^{-1}\left(
F(K_1)\,\sharp\, F(L_1)\right),F^{-1}\left(F(K_2)\,\sharp\, F(L_2)\right)\right)\\
&\le 4\delta_{LP}\left(F(K_1)\,\sharp\, F(L_1),F(K_2)\,\sharp\, F(L_2)\right)\\
&\leq 8\max\left\{\delta_{LP}\left(F(K_1),F(K_2)\right),\delta_{LP}\left(
F(L_1),F(L_2)\right)\right\}\\
&\leq 32\max\left\{\delta_{LP}\left(K_1,K_2\right),\delta_{LP}\left(
L_1,L_2\right)\right\}<\ee.
\end{align*}
This proves the claim.

Next, we claim that when $n\ge 3$, the operation $*$ is not $GL(n)$ covariant.  To this end, let $a,b\ge 0$ and let $K=L=[-1/2,1/2]^n$, so that $S(K)=S(L)=2n$ and
$$aK\,\sharp\,bL= \left[-\frac{\left(a^{n-1}+b^{n-1}\right)^{1/(n-1)}}{2}
,\frac{\left(a^{n-1}+b^{n-1}\right)^{1/(n-1)}}{2}\right]^n.$$
Then $F(K)$ and $F(L)$ are rotations of $[-1/2,1/2]^n$ by an angle of $2n$ around the origin in the $\{x_1,x_2\}$-plane, so $F(K)\,\sharp\,F(L)$ is a rotation of $\left[-2^{-1+1/(n-1)},2^{-1+1/(n-1)}\right]^n$ by an angle of $2n$ around the origin in the $\{x_1,x_2\}$-plane.  Since $S\left(F(K)\,\sharp\,F(L)\right)=4n$, $K*L$ is a rotation of $\left[-2^{-1+1/(n-1)},2^{-1+1/(n-1)}\right]^n$ by an angle of $-2n$ around the origin in the $\{x_1,x_2\}$-plane.  This shows that $K*L\neq aK\,\sharp\,bL$. In view of Theorem~\ref{thmBlaschke}, this proves the claim.

Finally, $*$ has the limit identity property with respect to the L\'{e}vy-Prokhorov metric $\delta_{LP}$.  To prove this, let $s>0$ and $K\in {\mathcal{K}}^n_{os}$. Then $F(sB^n)=sB^n$ and hence by (\ref{eqBEx3}),
$$K*(sB^n)=F^{-1}\left(F(K)\,\sharp\,(sB^n)\right).$$
Since $F(K)$ is a rotation of $K$ by an angle of $S(K)$ around the origin in the $\{x_1,x_2\}$-plane, $F(K)\,\sharp \, (sB^n)$ is a rotation of $K\,\sharp \, (sB^n)$ by an angle of $S(K)$ around the origin in the $\{x_1,x_2\}$-plane.  (Here we used the rotation covariance of Blaschke addition and the rotation invariance of $sB^n$.)
Moreover,
\begin{equation}\label{exsurf2}
S\left(K\,\sharp\,(sB^n),\cdot\right)=S(K,\cdot)
+S\left(sB^n,\cdot\right).
\end{equation}
In particular, we have
$$
S\left(K\,\sharp\,(sB^n)\right)=S(K)
+S\left(sB^n\right)=S(K)+n\kappa_ns^{n-1},
$$
so $K*(sB^n)$ is a rotation of $K\,\sharp \, (sB^n)$ by an angle of $-n\kappa_ns^{n-1}$ around the origin in the $\{x_1,x_2\}$-plane. Hence $\delta\left(K*(sB^n), K\,\sharp \, (sB^n)\right)\to 0$ as $s\to 0$ and from Proposition~\ref{SecondLPHaus}, we conclude that
$\delta_{LP}\left(K*(sB^n), K\,\sharp \, (sB^n)\right)\to 0$ as $s\to 0$.  Also, from (\ref{exsurf2}), we obtain $\delta_{LP}\left(K\,\sharp \, (sB^n), K\right)\to 0$ as $s\to 0$.  It follows that $\delta_{LP}\left(K*(sB^n), K\right)\to 0$ as $s\to 0$. As the operation $*$ is commutative, this proves that $*$ has the limit identity property with respect to the L\'{e}vy-Prokhorov metric $\delta_{LP}$.
}
\end{ex}

\begin{ex}\label{BlasEx3}
{\em The operation $*:\left(\cK^n_{os}\right)^2\to \cK^n_{os}$, $n\ge 3$, defined by $K*L=K\,\sharp\,2L$ for $K,L\in \cK^n_{os}$ is uniformly continuous in the L\'{e}vy-Prokhorov metric ${\delta}_{LP}$ and $GL(n)$ covariant, by Theorem~\ref{thmBlaschke}, but clearly does not have the limit identity property.}
\end{ex}

\end{document}